\documentclass[11pt,a4paper]{article}
\usepackage[margin=.8in]{geometry}
\usepackage[dvips]{graphicx}
\usepackage{amsmath}
\usepackage{amsthm}
\usepackage{amssymb}
\usepackage{verbatim}
\usepackage{algorithmic}
\usepackage{algorithm}
\usepackage{enumerate}
\usepackage{framed}
\usepackage{marginnote}
\usepackage[dvipsnames]{xcolor}
\usepackage{hyperref}
\hypersetup{
    colorlinks=true,
    linkcolor=blue,
    filecolor=magenta,      
    urlcolor=cyan,
}

\usepackage{tikz}
\usepackage{calc}
\usetikzlibrary{decorations.markings,calc, shapes, arrows.meta, positioning,shapes.multipart}
\tikzstyle{every picture}=[line width=.65pt,minimum size=3pt,every label/.append style={font=\small},label distance=-2pt]
\tikzstyle{every node}=[font=\small]
\tikzset{>=stealth}
\tikzstyle{vtx}=[circle,draw,thick,fill=black!50]
\tikzstyle{vtxr}=[circle,draw,thick,fill=red!50]
\tikzstyle{vtxp}=[circle,draw,thick,fill=pink!50]
\tikzstyle{vtxb}=[circle,draw,thick,fill=blue!30]
\tikzstyle{vtxgre}=[circle,draw,thick,fill=green!60]

\usepackage[maxnames=50,backend=biber,giveninits=true]{biblatex}
\usepackage{csquotes}
\allowdisplaybreaks

\newtheorem{theorem}{Theorem}[section]

\newtheorem{lemma}[theorem]{Lemma}
\newtheorem{corollary}[theorem]{Corollary}

\newtheorem{definition}[theorem]{Definition} 

\newtheorem{observation}[theorem]{Observation}
\newtheorem{remark}[theorem]{Remark}

\newcommand{\type}{\mathrm{type}}
\mathchardef\mhyphen="2D

\newcommand{\B}{\mathcal{B}}

\newcommand{\R}{\mathbb{R}}

\newcommand{\conv}{\operatorname{conv}}

\newcommand{\Z}{\mathbb{Z}}

\newcommand{\rk}{\operatorname{rk}}

\newcommand{\lpr}{\operatorname{lpr}}
\newcommand{\sk}{\operatorname{sk}}

\title{Binary extended formulations and sequential convexification}
\author{Manuel Aprile \and Michele Conforti \and Marco Di Summa}

\begin{document}
\maketitle

\begin{abstract}
A binarization of a bounded variable $x$ is a linear formulation with variables $x$ and additional binary variables $y_1,\dots, y_k$, so that integrality of $x$ is implied by the integrality of  $y_1,\dots, y_k$.

A binary extended formulation of a polyhedron $P$ is obtained by adding to the original description of $P$ binarizations of some of its variables. In the context of mixed-integer programming, imposing integrality on 0/1 variables rather than on general integer variables has interesting convergence properties and has been studied both from the theoretical and from the practical point of view.

We propose a notion of \emph{natural} binarizations and binary extended formulations, encompassing all the ones studied in the literature. We give a simple characterization of the vertices of such formulations, which allows us to study their behavior with respect to  sequential convexification. 
In particular, given a binary extended formulation and 
 one of its variables $x$, we study a parameter  that measures the progress made towards ensuring the integrality of $x$ via application of sequential convexification. We formulate this parameter, which we call rank, as the solution of a set covering problem and express it exactly for the classical binarizations from the literature.

\end{abstract}

\section{Introduction}

A  technique that has been widely investigated in combinatorial optimization and integer  programming is to formulate a given problem in a higher-dimensional space, comprising  the original, natural variables along with a number of additional variables. This general idea is motivated by the fact that these formulations, called  {\em extended formulations}, may be easier to deal with, as they may have a small number of inequalities or may exhibit a structure that is not evident in the original space. (See, e.g., \cite{conforti2010extended} for a survey on extended formulations.)

One way to obtain an extended formulation of a mixed-integer program is to reformulate each bounded integer variable by means of a set of binary variables, which are then relaxed to lie in the interval $[0,1]$. In this case, the formulation is called a {\em binary extended formulation}. The use of binary variables is motivated both by theoretical and computational reasons, as we will discuss below.

The idea of using binary reformulations dates back at least to Glover \cite{glover1975improved}, who proposed methods to linearize nonlinear conditions by means of binary variables.  In the context of mixed-integer linear programming, different binarization techniques have been explored in the literature. We now mention some of the most relevant methods. Below we assume that $x$ is an integer variable whose range of possible values is $\{0,\dots,k\}$.
\begin{itemize}
\item Sherali and Adams \cite{sherali2013reformulation} suggested to model $x$ by using $k$ binary variables $y_1,\dots,y_k$ as follows: $x=\sum_{i=1}^ki\cdot y_i$ with $\sum_{i=1}^ky_i\le1$. This is sometimes referred to as the {\em full binarization}.
\item Assuming for the sake of simplicity that $k=2^t-1$ for some integer $t\ge0$, Owen and Mehrotra \cite{owen2002value} proposed a reformulation with $t$ binary variables of the form $x=\sum_{i=1}^t2^{i-1}y_i$. This is called the {\em compact} or {\em logarithmic binarization}.
\item Roy \cite{roy2007binarize} suggested to write $x=\sum_{i=1}^k y_i$ with $y_1\ge\dots\ge y_k$. Following \cite{dash2018binary}, we call this the {\em unary binarization}.
\end{itemize}

In all the above examples, the integrality requirement on $x$ can be relaxed, as it is implied by the integrality of $y$. When variables $y$ are only required to be in the interval $[0,1]$, one obtains a polytope $P$ in the variables $x$ and $y$ with the following property: the projection of the set $\{(x,y)\in P:y_i\in\{0,1\}\:\forall i\}$ is the range of possible values of $x$, i.e., $\{0,\dots,k\}$. Furthermore, the rule used to model $x$ by means of binary variables is independent of the constraints on the original $x$ variables. A polytope of this type is called a {\em binarization polytope}, or simply a {\em binarization}, and a binary extended formulation of a polyhedron $P$ can be obtained by adding to the description of $P$ a binarization for each integer variable of $P$. (We postpone the formal definitions to Section \ref{sec:bin}.) 
We also remark that in the above binarizations $x$ is linked to the $y$ variables by means of a linear equation. In this situation, we will call the binarization {\em affine}. However, binarizations in general may not satisfy this property, and indeed most of our results apply to non-affine binarizations, as well.\medskip

Modelling with binary variables allows the use of {\em canonical disjunctions} of the form ``$y_i=0$ or $y_i=1$'' (where $y_i$ is a binary variable), which induce faces of the formulation. 
This property implies that, assuming that there are, say, $d$ binary variables and all other variables are continuous (as is the case for a binary extended formulation), the integer hull can be obtained in $d$ steps by taking the convex hull of the faces induced by $y_i=0$ and $y_i=1$, and iterating over all binary variables. (This is Balas' sequential convexification technique \cite{balas1998disjunctive}, which we review in Section \ref{sec:seq-conv}). In contrast, general split disjunctions in the original space (i.e., disjunction of the form ``$ax\le\beta$ or $ax\ge\beta+1$" for some integer vector $a$ and some integer number $b$) do not necessarily yield the integer hull in a finite number of steps. (See, e.g., \cite{conforti2014integer} for more details on split disjunctions.)

From the practical point of view, binary mixed-integer problems tend to be easier to solve, compared to problems with general integer variables. However, there is a trade-off with the increase of computational effort due to the addition of many variables, and indeed, there are contrasting results on this point. For instance, Owen and Mehrotra \cite{owen2002value} showed that, restricting sequential convexification to the binary variables of a single logarithmic binarization, one needs to convexify all such variables in order to reduce the range of the original integer variable by half: this suggests that  using this type of binarization out-of-the-box in a branch-and-bound framework might result in larger trees than those one would obtain in the original space. 
On the other hand, Roy \cite{roy2007binarize} proved that if one uses the unary binarization and generates suitable cuts that can be expressed in terms of the original variables, then the resulting cutting planes are in general stronger than those that one would produce directly in the original space. Bonami and Margot \cite{bonami2015cut} showed a result in the same spirit, where only canonical disjunctions are applied.
\medskip

Roy \cite{roy2007binarize} gave a definition of binarization for the affine case. A generalization of this notion, that we also adopt in this paper, was given by Dash, G\"unl\"uk and Hildebrand \cite{dash2018binary}. They studied the behavior of binarizations with respect to the split closure, and showed that a class of binarizations called unimodular is optimal in this sense. Moreover, they showed that using appropriate binarizations can lead to smaller branch-and-bound trees, in contrast to what is suggested by \cite{owen2002value}. 

The unimodular binarizations from \cite{dash2018binary}, although optimal with respect to the split closure, have a large number of variables (exactly $k$ if the range of the original variable is $\{0,\dots,k\}$). Dash et al. \cite{dash2018binary} suggest that it would be interesting to investigate the strength of binarizations with fewer variables.
\medskip

The work described above motivates the question: how can we evaluate binarizations and determine which ones are ``better''? 
In this paper we approach the question from an original perspective, centered around the technique of sequential convexification. We remark that the performance of sequential convexification on a binary extended formulation depends heavily on the polytope of interest, hence different choices of binarizations might be preferable for different polytopes: 
in spite of that, we offer a notion of rank that depends on a single binarization (and not on the original polytope) and models its behaviour with respect to sequential convexification. 

Our contributions can be summarized as follows:
\begin{itemize}
    \item In Section \ref{sec:seqcon-bin} we observe that the smallest number of disjunctions that remove a given set of fractional vertices is the solution of a set covering problem. We will build on this observation throughout the paper.
    \item In Section \ref{sec:vertices} we characterize the vertices of a binary extended formulation and their projections on the original space, provided the binarizations satisfy a ``natural" assumption. This applies, in particular, to the three fundamental  binarizations mentioned above.
    \item In Section \ref{sec:rank} we define a notion of rank of a binarization that is related to the lift-and-project rank of a polytope (defined in Section \ref{sec:seq-conv}). Informally, our rank models the problem of 
removing a set of fractional vertices of a binary extended binarization by performing the smallest number of disjunctions on the 0/1 variables of a fixed single binarization. We show that, under the aforementioned natural assumptions, our rank is the solution of a set covering problem that depends on the binarization only, and not on the rest of the binary extended formulation.
\item In Section \ref{sec:rank} we also provide explicit formulas for the rank of the three  binarizations mentioned above, and we especially focus on the logarithmic binarization, which has the smallest possible number of variables: we show that it has minimum rank among all binarizations with the same ``structure'' (in particular, with the same number of variables).
\end{itemize}
 
 We believe that our notion of rank leads to an educated guess on which binarizations yield better computational results, as it measures ``partial progress'' towards obtaining the mixed integer hull of the polytope of interest. This is motivated further in Section \ref{sec:conclusion}, where we review our results and propose further directions of research.

\section{Sequential convexification and binarizations}\label{sec:seqcon-bin}

\subsection{Vertex interpretation of sequential convexification}\label{sec:seq-conv}

We recall a fundamental result of Balas \cite{balas1998disjunctive} on sequential convexification of a polytope contained in the unit cube, and  give an elementary proof  that will be useful throughout the paper.

Given a polytope $Q$, we denote with $V(Q)$ the set of its vertices. 
If $Q\subseteq [0,1]^h\times \R^{n-h}$ and we fix any variable $x_i$ with $i\in [h]$ (where we use the notation $[h]:=\{1,\dots,h\}$), we define 
$$Q_{x_i}:=\conv\left(\{x\in Q:x_i=0\} \cup \{x\in Q:x_i=1\}\right).$$
We now give a ``vertex'' proof of the sequential convexification theorem of Balas \cite{balas1998disjunctive}.
\begin{theorem} [Sequential convexification theorem of Balas] \label{thm:balas}
For every polytope  $Q\subseteq [0,1]^h\times \R^{n-h}$, we have  $$\conv\{x\in Q : x_i\in \{0,1\}\:\forall i\in [h]\}=(((Q_{x_1})_{x_2})\dots )_{x_h}.$$
\end{theorem}

\begin{proof} Since $x_1=0$ defines a face of $Q$, say $F_0$, we have $V(F_0)=\{x\in V(Q):x_1=0\}$. Therefore, if we let $V_1$ be the set of vertices of $Q$ with $0<x_1<1$, then  $V(Q_{x_1})=V(Q)\setminus V_1=\{x\in V(Q):x_1\in \{0,1\}\}$. By iterating this argument, we obtain $V((((Q_{x_1})_{x_2})\dots )_{x_h})=\{x\in V(Q):x_i\in \{0,1\}\:\forall i\in [h]\}$, and this proves the result.
\end{proof}

While above theorem is stated for polytopes,  it holds for a (possibly unbounded) polyhedron $Q\subseteq [0,1]^h\times \R^{n-h}$. Indeed, by intersecting $Q$ with the orthogonal complement of its lineality space (so the resulting polyhedron  contains vertices whenever $Q$ is nonempty), and observing  that the recession cone of $Q$ is contained in $\{0\}^h\times \R^{n-k}$, the above proof can be carried out in the same way. We also remark that Theorem \ref{thm:balas} implies that the actual order of the sequence of convexifications is irrelevant.
\medskip

Note that while Balas extended formulation for the convex hull of the union of polytopes provides an inequality description of $Q_{x_1}$ (and of $(((Q_{x_1})_{x_2})\dots )_{x_h}$) in an extended space (see \cite[Section 4.9]{conforti2014integer}), the inequality description of $Q_{x_1}$ in the original space may be very complicated with respect to the original description of the polytopes  \cite{conforti2020balas}. However, the above proof is based on the fact  that a vertex description of $\conv\{x\in Q : x_i\in \{0,1\}\:\forall i\in [h]\}$ can be derived immediately from $V(Q)$.
\medskip

Let $Q\subseteq [0,1]^h\times \R^{n-h}$ be a polytope. The \emph{lift-and-project rank} $\lpr(Q)$ (with respect to variables $x_1,\dots,x_h$) is the minimum integer $k$ such that there are variables $x_{i_1},\dots, x_{i_k}$, with $i_1,\dots, i_k\in [h]$, satisfying $((Q_{x_{i_1}})\dots)_{ x_{i_k}}=\conv\{x\in Q : x_i\in \{0,1\}\:\forall i\in [h]\}$. (See, e.g., \cite[Section 5.4]{conforti2014integer})

Theorem \ref{thm:balas} shows that $\lpr(Q)\le h$. Furthermore,  one can derive a combinatorial interpretation of the lift-and-project rank as the solution of a minimum  hitting set (or set covering) problem.
Indeed, let $V_F$ be the set of vertices in $V(Q)$ containing at least one fractional component among $x_1,\dots,x_h$, and let $\mathcal{F}\subseteq 2^{[h]}$ be the family that contains, for each vertex in $V_F$, the subset of indices of the fractional components of that vertex (where only components $x_1,\dots,x_h$ are considered). 
Then $\lpr(Q)$ is exactly the minimum size of a hitting set of $\mathcal{F}$, i.e., a minimum subset of $[h]$ that intersects all the sets in $\mathcal F$. Thus, if $A_{Q}$ denotes the $|V_F|\times h$ $0/1$-matrix whose rows are the incidence vectors of the members of $\mathcal F$, we have the following. (We use $\mathbf1$ to denote the all-one vector of appropriate dimension.)

\begin{observation}\label{obs:set-covering}
Let $Q\subseteq [0,1]^h\times \R^{n-h}$ be a polytope.
Then $$\lpr(Q)=\min\{ \mathbf {1}z:A_{Q}z\ge \mathbf{1},\, z\in \{0,1\}^h\}.$$
\end{observation}

As it is customary, we refer to the integer program above as a \emph{set covering problem}.

Sequential convexification can be applied, in principle, also when $x_1,\dots,x_h$ are not bounded between 0 and 1. Indeed, one can consider a canonical disjunction of the form ``$x_i\le\alpha$ or $x_i\ge\alpha+1$'' for some $\alpha\in\Z$ and update $Q$ with $\conv(\{x\in Q:x_i\le\alpha\}\cup\{x\in Q:x_i\ge\alpha+1\})$. However, iterating this procedure does not always lead to $\conv\{x\in Q:x_i\in\Z\:\forall i\in[h]\}$. For instance, if $Q\subseteq \R^2$ is defined as the convex hull of $\{(0, 0),(1.5, 1),(2, 2),(1, 1.5)\}$, one can see (\cite[Theorem 2.4]{basu2020complexity}) that the integer hull of $Q$ cannot be obtained by applying a finite sequence of canonical split disjunctions.

\subsection{Binarizations}\label{sec:bin}

Given a set $S\subseteq\{(x,y):(x,y)\in\R^p\times\R^d\}$, we denote by $\pi_x(S)$ the orthogonal projection of $S$ on the $x$-space, i.e., $\pi_x(S)=\{x\in\R^p:(x,y)\in S\mbox{ for some $y\in\R^d$}\}$; $\pi_y(S)$ is defined similarly.

  A polytope $B\subseteq\{(x,y):(x,y)\in \R\times [0,1]^d\}$ is a \emph{binarization} of variable  $x$ in the range $\{0,\dots,k\}$ (for some positive integer $k$) if
  \begin{equation}\label{eq:binarization}
  \pi_{x}(\{(x,y)\in B:y\in \{0,1\}^d\})=\{0,\dots,k\}.
  \end{equation}
  This implies that $d\ge \lceil \log (k+1)\rceil$.
The following was observed in \cite{dash2018binary}.

\begin{observation}\label{obs:onex}
 Let $B\subseteq \R\times [0,1]^d$ be a binarization of variable  $x$ in the range $\{0,\dots,k\}$. Given  $y\in \{0,1\}^d$, there is at most one $x$ such that $(x,y)\in B$. 
If such $x$ exists, then $x\in \{0,\dots,k\}$ and $(x,y)\in V(B)$.
\end{observation}

\begin{proof}   The first statement and the fact that  $x\in \{0,\dots,k\}$  follow from \eqref{eq:binarization}, because if $B$ contains $(x^1,y)$ and $(x^2,y)$ with $x^1\ne x^2$, then $B$ contains $(x^*,y)$ with $x^*\not\in \Z$. 
To show the last statement, assume that a point $(x,y)\in B$ with $y\in \{0,1\}^d$ is not a vertex of $B$. Then $(x,y)$ is in the relative interior of a segment with endpoints $(x^1,y^1)$ and $(x^2,y^2)$ in $B$.  Since $y\in \{0,1\}^d$, $y^1=y^2=y$. Again, the above argument shows that this is possible only if $x^1=x^2=x$.
\end{proof}

Given a polytope $P\subseteq [0,k]^p\times \R^{n-p}$, with variables $x_1,\dots,x_n$, for every $i\in[p]$ let $B_i\subseteq\{(x_i,y_i):(x_i,y_i)\in \R\times [0,1]^d\}$ be a binarization of $x_i$. We denote the components of $y_i$ by $y_{i_1},\dots,y_{i_d}$.
 The \emph{binary extended formulation} $Q$ of $P$ with binarizations $B_1,\dots, B_p$ is 
$$
Q:=\{(x,y)\in \R^n\times [0,1]^{pd}: x\in P,\, (x_i,y_i) \in B_i \:\forall i\in [p]\}.
$$
Notice that $Q$ is an extended formulation of $P$, i.e., $\pi_x(Q)=P$. Since $\pi_{x_i}(\{(x_i,y_i)\in B_i:y_i\in \{0,1\}^d\})=\{0,\dots,k\}$ for $i\in [p]$, we have that
$$\{x\in P:x_i\in \{0,\dots,k\}\:\forall i\in[p]\}=\pi_x\left(\{(x,y)\in Q: y_{i}\in\{0,1\}^d\: \forall i\in [p]\}\right).$$
Thus the mixed-integer set in the $x$-space on the left-hand side is lifted to a mixed $0/1$-set in an extended space.

The above observation, together with Theorem \ref{thm:balas}, implies that by applying sequential convexification to the $y$ variables of $Q$ one obtains the mixed integer hull of $P$ (with respect to variables $x_1,\dots, x_p$). In particular, we have the following persistency property:

\begin{observation}\label{obs-persistency} 
Given a polytope $P\subseteq [0,k]^p\times \R^{n-p}$ and a binary extended formulation $Q$ of $P$ 
with binarizations $B_i\subseteq\{(x_i,y_i):(x_i,y_i)\in \R\times [0,1]^d\}$ for $i\in [p]$, consider a variable $y_{\bar{\imath}\bar{\jmath}}$. If no point in $V(Q)$ (equivalently, in $\pi_x(V(Q))$) 
satisfies  $\alpha <x_{\tilde \imath}<\alpha +1$ for some $\tilde{\imath}\in [p]$ and $\alpha\in\Z$, then no point in $
V(Q_{y_{\bar{\imath}\bar{\jmath}}})$ (equivalently, in $\pi_x(
V(Q_{y_{\bar{\imath}\bar{\jmath}}}))$) satisfies  $\alpha <x_{\tilde\imath}<\alpha +1$. 
\end{observation}

\begin{proof}  Since  $Q_{y_{\bar{\imath}\bar{\jmath}}}$ is the convex hull of two faces of $Q$, we have that $V(Q_{y_{\bar{\imath}\bar{\jmath}}})\subseteq V(Q)$. It follows that   $\pi_x(V(Q_{y_{\bar{\imath}\bar{\jmath}}}))\subseteq \pi_x(V(Q))$, and this proves the observation. \end{proof}

Indeed, the above holds for any hereditary property $\mathcal P$, where ``hereditary'' means: ``If $S\subseteq V(Q)$ satisfies  $\mathcal P$ and $S'
\subseteq S$, then $S'$ satisfies  $\mathcal P$''. We stress that the convergence to the mixed integer hull of $P$ implied by Observation \ref{obs-persistency} is not always possible in the original space, as the example at the end of Section \ref{sec:seq-conv} shows.

We remark that, since different variables $x_i$ can have different ranges and different binarizations, one can consider binary extended formulations where each binarization $B_i$ has a different dimension $d_i$. For simplicity of notation, we stick to the case where $d_i=d$ for each $i\in[p]$. However, all our results, with straightforward modifications, hold in the general setting, as well.




\subsection{Properties of binarizations}\label{sec:properties}
In this section we define several properties that a binarization can have, and some specific types of binarizations that have been used in the literature and are investigated further in this paper. Some of the properties below have been defined in \cite{dash2018binary}.

In the following we assume that $B\subseteq\{(x,y):(x,y)\in\R\times[0,1]^d\}$ is a binarization of $x$ in the range $\{0,\dots,k\}$.

\begin{definition}\label{def:natural}
A binarization $B$  is \emph{natural} if $x$ is integer for any vertex $(x,y)$ of $B$.  
\end{definition}

In other words, a binarization $B$ is natural if and only if $\pi_x(V(B))=\{0,\dots,k\}$. 

\begin{definition}
A binarization $B$ is \emph{integral} if $B=\conv(B\cap (\Z\times \{0,1\}^d))$.
\end{definition}

It follows from the definition of binarization that the above condition is equivalent to $B=\conv(B\cap (\R\times \{0,1\}^d))$.

\begin{definition}
A binarization $B$ is \emph{exact} if for every $x\in \{0,\dots,k\}$ there exists exactly one $y\in \{0,1\}^d$ such that  $(x,y)$ in $B$.
\end{definition}

\begin{observation}\label{obs:twox}
Let $B$ be an exact binarization. Then $B$ is natural if and only if  $B$  is integral.
\end{observation}
\begin{proof} If  $B$ is integral, then $B$ is obviously natural. For the converse,  let $(x,y)$ be a vertex of $B$, where $B$ is exact and natural. As $B$ is natural, $x\in \Z$. Since $B$ is exact,  $y$ is unique and therefore $y\in \{0,1\}^d$.
\end{proof}


If $B$  is a  binarization which is exact and natural (or equivalently exact and integral), we say that $B$ is {\em perfect}.
 Notice that if $B$ is integral, then $B$ is natural but may not be exact. 

\begin{definition}\label{def:affine}
A binarization $B$ is \emph{affine} if $x=f(y)$ for all $(x,y)\in B$, where $f: \R^d\rightarrow \R$ is an affine function.
\end{definition}

We recall some well-known binarizations that we have already mentioned in the introduction:
$$B^U(d)=\{(x,y)\in \R\times [0,1]^d: x= \textstyle\sum_{i=1}^d y_i,\, 1\geq y_1\geq \dots \geq y_d\geq 0\};$$
$$B^F(d)=\{(x,y)\in \R\times [0,1]^d: x= \textstyle\sum_{i=1}^d i\cdot y_i, \, \textstyle\sum_{i=1}^d y_i\leq 1\};$$
$$B^L(d)=\{(x,y)\in \R\times [0,1]^d: x=\textstyle\sum_{i=1}^{d} 2^{i-1}y_i \}.$$
$B^U(d)$ is known as the \emph{unary} binarization; $B^F(d)$ is the \emph{full} binarization; $B^L(d)$ is the \emph{logarithmic} (or compact) binarization. Notice that in $B^U(d)$ and $B^F(d)$ the number of $y$ variables $d$ is also the upper bound of the range of the $x$ variable (i.e., $x\in\{0,\dots,d\}$), whereas in $B^L(d)$ the range of $x$ is $\{0,\dots, 2^d-1\}$. Hence $B^L(d)$ uses a logarithmic number of variables with respect to the length of the range of $x$ and, as defined here, requires $x$ to be a power of 2. The general case of an arbitrary range of $x$, also studied in \cite{dash2018binary}, is treated in Section \ref{sec:trunc}.
We remark that $B^U(d)$, $B^F(d)$ and $B^L(d)$ are all  perfect binarizations (hence natural, integral and exact) and are all affine.



\section{The vertices of a binary extended formulation}\label{sec:vertices}

Let $Q$ be a binary extended formulation of a polytope $P\subseteq [0,k]^p\times \R^{n-p}$. We now  describe the vertices of $Q$ and their projections on the original space, provided that the associated binarizations are natural. 


 Given 
$I\subseteq [p]$  and $\alpha\in \{0,\dots,k\}^{I}$, we define the subspace $G_{I,\alpha}:=\{x\in \R^n:x_i=\alpha_{i}\:\forall i\in I\}$.

\begin{theorem}\label{thm:verticesQ}
 Let $P\subseteq [0,k]^p\times \R^{n-p}$ be a polytope and let $Q\subseteq\R^n\times [0,1]^{pd}$ be a binary extended formulation of $P$ with binarizations $B_i\subseteq \R\times [0,1]^d$ for $i\in [p]$ that are natural. 
Then $(\bar{x},\bar{y}) \in \R^n\times [0,1]^{pd}$ is a vertex of $Q$ if and only if there exist a subspace $G_{I,\alpha}$ and a face $F$ of $P$ of dimension $|I|$ such that:
\begin{enumerate}[i)]
\item $F\cap G_{I,\alpha}=\{\bar{x}\};$
\item $(\bar{x}_i,\bar{y}_i)\in V(B_i)\:\forall i\in I$;
\item $(\bar{x}_i,\bar{y}_i)\in V(B_i\cap\{x_i=\bar{x}_i\}) \:\forall i\in [p]\setminus I$.
\end{enumerate}
\end{theorem}

    \begin{proof}
    We work with the linear inequality description of $Q$ given by some fixed linear inequality descriptions of $P$ and $B_i$, for $i\in [p]$.\medskip
     
     We first prove the ``only if part''.
        Given $(\bar{x},\bar{y})\in V(Q)$, define $\tilde{I}=\{i\in [p]:\bar x_i\in  \{0,\dots,k\}\}$, and let $F$ be the face of $P$ of minimum dimension containing $\bar{x}$. We will show that $(\bar{x},\bar{y})$ satisfies i)--iii) by using $F$ and $G_{I,\alpha}$ for an appropriate subset $I$ of $\tilde{I}$.
        
        Let $q$ be the dimension of $F$, and consider a basis $\B$ of $Q$ defining $(\bar{x}, \bar{y})$ that contains $n-q$   tight inequalities from the system defining $P$. Such a basis exists because we chose $F$ of minimum dimension.

        Notice that
    $\B$ contains at least $d$ tight inequalities  from the system defining each $B_i$ for $i\in [p]$, as otherwise some variable $y_{ij}$ of $B_i$ would not be ``fixed'' by $\B$. Furthermore, we claim that
    $(\bar{x}_i,\bar{y}_i)\in  V(B_i\cap\{x_i=\bar{x}_i\})$ for each $i\in[p]$: Indeed, consider a convex combination of points $(\bar{x}_i, y^{(1)}_i), \dots, (\bar{x}_i, y^{(t)}_i)\in V(B_i\cap\{x_i=\bar{x}_i\})$ which is equal to $(\bar{x}_i,\bar{y}_i)$. Then one can obtain $(\bar{x},\bar{y})$ as a convex combination of points $(\bar{x},\bar{y}^{(1)}), \dots, (\bar{x},\bar{y}^{(t)})$ where each $\bar{y}^{(j)}$ is equal to $y^{(j)}$ in coordinates $y_i$, and to $\bar{y}$ in  all other coordinates. Hence, the fact that $(\bar{x},\bar{y})\in V(Q)$ proves our claim.
    
    Moreover, for $i\notin \tilde{I}$, $\B$ contains exactly $d$ tight inequalities 
    from the system defining $B_i$: indeed, it cannot contain $d+1$ tight inequalities as, $\bar{x}_i$ being fractional and $B_i$ being natural, $(\bar{x}_i,\bar{y}_i)$ cannot be a vertex of $B_i$. But then, since $|\B|=n+pd$, there is a subset $I\subseteq \tilde{I}$ of size $q$ such that, for $i\in I$, $\B$ contains $d+1$ tight inequalities from the system defining $B_i$, i.e., $(\bar{x}_i,\bar{y}_i)\in V(B_i)$. This also implies that for every $i\in I$ the equation $x_i=\bar{x}_i$ is linearly independent from the equations defining $F$. Hence  $F\cap G_{I,\alpha}=\{\bar{x}\}$, where $\alpha_i=\bar{x}_i$ for $i\in I$.      This proves that $(\bar{x},\bar{y})$ satisfies i)--iii).\medskip
    
    We now prove the ``if'' part.
    Let $(\bar{x},\bar{y})$ satisfy i)--iii) for some $F, I,\alpha$. Notice that $(\bar{x},\bar{y})\in Q$, and it satisfies at equality the following set $\B$ of inequalities from the system defining $Q$: $n-|I|$ linearly independent inequalities defining $F$; $d+1$ linearly independent inequalities from each $B_i$ with $i\in I$; $d$ linearly independent inequalities from each $B_i$ with $i\in[p]\setminus I$. In order to conclude that $(\bar{x},\bar{y})\in V(Q)$, it suffices to show that all these tight inequalities are linearly independent (implying that $\B$ is a basis). Now, since the $B_i$'s have disjoint sets of variables, if $\B$ is not a basis then there is an equation among those defining $F$ that is linearly dependent with some others. However notice that, since $ F\cap G_{I,\alpha}=\{\bar{x}\}$, the system of equations defining $F$ is equivalent to the system $x_i=\bar{x}_i$ for $i\in [p]\setminus I$. Hence, if $\B$ is not a basis, there is $i \in[p]\setminus I$ such that the equation $x_i=\bar{x}_i$ is linearly dependent with the equations of $\cal B$ picked from $B_i$. But this contradicts the fact that $(\bar{x}_i,\bar{y}_i)\in V(B_i\cap\{x_i=\bar{x}_i\})$. 
    \end{proof}

\begin{theorem}\label{thm:vertices}
Let $P\subseteq [0,k]^p\times \R^{n-p}$ be a polytope and let $Q$ be a binary extended formulation of $P$ with binarizations $B_i\subseteq \R\times [0,1]^d$ for $i\in [p]$ that are natural. Then $\bar{x}\in \R^n$ is a point in  $\pi_x(V(Q))$ if and only if there exist a subspace $G_{I,\alpha}$ and a face $F$ of $P$ of dimension $|I|$ such that
$F\cap G_{I,\alpha}=\{\bar{x}\}$.

\end{theorem}

\begin{proof}
The ``only if''part follows from Theorem \ref{thm:verticesQ}. 

Assume now $F\cap G_{I,\alpha}=\{\bar{x}\}$. For $i\in I$, since $\bar{x}_i\in \{0,\dots,k\}$, by Observation \ref{obs:onex} there exists $\bar{y}_i\in \{0,1\}^{d}$ such that $(\bar{x}_i,\bar{y}_i)\in V(B_i)$. For $i\in[p]\setminus I$, $\bar{x}_i$ can be obviously extended to a point in  $V(B_i\cap\{x_i=\bar{x}_i\})$.  Therefore the ``if'' part also follows from Theorem \ref{thm:verticesQ}.
\end{proof}

Here are some consequences of the above theorems.
\begin{itemize}
\item The proof of the ``if'' parts in Theorems \ref{thm:verticesQ} and \ref{thm:vertices}  does not use the fact that the binarizations are natural, hence these implications hold for any binary extended formulation. However, the assumption that the binarizations are natural is needed for the ``only if'' parts, as Remark \ref{rem:nonnatural} in Section \ref{sec:example} shows. 
\item Given a polytope $P$ and a binary extended formulation $Q$ of $P$ as in Theorem \ref{thm:verticesQ},  the set $V(Q)$ (and its cardinality)  is a function of the associated binarizations $B_1,\dots, B_p$. However Theorem \ref{thm:vertices} implies that the projection of $V(Q)$ on the $x$-space is independent of the  binarizations, as long as they are natural. 
    \item     In particular, Theorem \ref{thm:vertices} implies that $\pi_x(V(Q))$ always contains $V(P)$ because given  $\bar x\in V(P)$, one can choose $I=\emptyset$ and $F=\{\bar x\}$. Moreover, when $p=n$,  $\pi_x(V(Q))$ also contains $\{x\in P: x_i\in \{0,\dots,k\} \:\forall i\in[n]\}$. This is because given $\bar x\in \{x\in P: x_i\in \{0,\dots,k\}$ one can choose $F=P$, $I\subseteq[n]$ such that $|I|=\dim(P)$ and $P\cap G_{I,\alpha}$ contains a single point (there must be such a subset $I$), and $\alpha_i=\bar{x}_i$ for all $i\in I$.
\end{itemize}

\subsection{An example}\label{sec:example}

 As a canonical disjunction on a binary variable is a split disjunction, the split rank of a polytope is less than or equal to the lift-and-project rank (we refer, e.g., to \cite{conforti2014integer} for precise definitions).

While the split rank of a rational polyhedron is always finite when all variables are required to be integer,
Cook, Kannan and Schrijver \cite{cook-kannan-schrijver} provided a simple example showing that this is not always the case in the mixed-integer setting. Indeed, the split rank of the polyhedron 
\begin{equation}\label{eq:pyramid}
P=\{(x_1,x_2,x_3)\in [0,2]^2\times \R:  hx_1+hx_2+x_3\leq 2h,\:
    x_3\leq 2h x_1,\:
    x_3\leq 2h x_2,\:
    x_3\ge 0
     \},
\end{equation}
where $h>0$, is infinite if only $x_1$ and $x_2$ are required to be integer. It is also known that the split rank of $P$ is $\Omega(\log h)$ when all the variables are restricted to be integer \cite{reverse-chatal}.
Here we use the above example to illustrate Theorems \ref{thm:verticesQ} and \ref{thm:vertices} on a binary extended formulation that uses natural binarizations, and to demonstrate the convergence of sequential convexification.

We have $V(P)=\left\{(0,0,0), (2,0,0),(0,2,0),(\frac{1}{2},\frac{1}{2},h)\right\}$.
We let $p=d=2$  and apply the unary binarization $B^U(2)$, which is natural: For $i=1,2$, let $B_i=\{(x_i,y_{i1},y_{i2})\in \R\times [0,1]^2: x_i=y_{i1}+y_{i2},\,  y_{i1}\geq y_{i2}\}$. Denoting by $Q$ the corresponding binary extended formulation and applying Theorem \ref{thm:vertices}, we find that $\pi_x(V(Q))$ consists of the points that are the unique elements in sets of the form $F\cap G_{I,\alpha}$, where $F$ is a $d$-face of $P$, $|I|=d$, and $\alpha\in\{0,1,2\}^I$ for some $d\in\{0,1,2\}$. In particular:
\begin{itemize}
\item  for $d=0$, we obtain the four vertices of $P$;
\item for $d=1$, we obtain again the three integer vertices of $P$, along with the points $(1,0,0)$, $(0,1,0)$, $(1,1,0)$, $(1,\frac13,\frac23h)$, $(\frac13,1,\frac23h)$; 
\item for $d=2$, we obtain again the six integer points of $P$.
\end{itemize}
Thus we have
\[
\pi_x(V(Q))=\left\{(0,0,0), (2,0,0),(0,2,0), \left(\frac{1}{2},\frac{1}{2},h\right), (1,0,0),(0,1,0), (1,1,0), \left(1,\frac{1}{3},\frac{2}{3}h\right), \left(\frac{1}{3},1,\frac{2}{3}h\right)\right\}.
\]
See Figure \ref{fig:pyramid}.

$V(Q)$ consists of the following points:
\[\begin{array}{c|ccccccc}
&x_1&x_2&x_3&y_{11}&y_{12}&y_{21}&y_{22}\\
\hline
v^1& 0 &  0 &   0 &  0 &  0 &  0 &  0 \\
v^2& 2 &  0 &   0 &  1 &  1 &  0 &  0 \\
v^3& 0 &  2 &   0 &  0 &  0 &  1 &  1 \\
v^4& 1/2 &1/2 & h &1/2 &  0 &1/2 &  0 \\
v^5& 1/2 &1/2 & h &1/2 & 0 &1/4 &1/4 \\
v^6& 1/2 &1/2 & h &1/4 &1/4 &1/2 &  0 \\
v^7& 1/2 &1/2 & h &1/4 &1/4 &1/4 &1/4 \\
v^8& 1 &  0 &   0 &  1 &  0 &  0 &  0 \\
v^9& 0 &  1 &   0 &  0 &  0 &  1 &  0 \\
v^{10}& 1 &  1  &  0 &  1 &  0 &  1 &  0 \\
v^{11}& 1 &  1 &   0 &1/2 &1/2 &  1 &  0 \\
v^{12}& 1 &  1 &   0 &  1 &  0 &1/2& 1/2 \\
v^{13}& 1& 1/3& 2h/3 &  1 &  0& 1/3 &  0 \\
v^{14}& 1& 1/3& 2h/3  & 1  & 0 &1/6& 1/6 \\
v^{15}& 1/3 &  1& 2h/3& 1/3 &  0 &  1 &  0 \\
v^{16}& 1/3 &  1& 2h/3& 1/6& 1/6 &  1 &  0
\end{array}\]


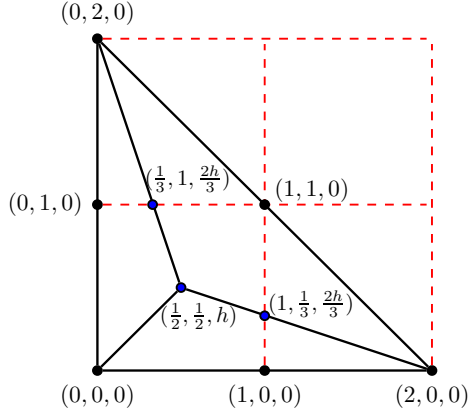
\begin{figure}
    \centering
\scalebox{1.1}{\begin{tikzpicture}
\tikzstyle{novtx}=[inner sep=0pt]
\tikzstyle{vtx}=[circle,draw, fill=black, minimum width = 1pt, inner sep=0pt]
\tikzstyle{vtxb}=[circle,draw, fill=blue, minimum width = 1pt, inner sep=0pt]

    \node[vtx] (v02) at (0,4) {}; 
    \node[xshift=0pt, yshift=7pt] at (v02.north) {\scalebox{0.8}{$(0,2,0)$}};
    \node[novtx] (v12) at (2,4) {};
    \node[novtx] (v22) at (4,4) {}; 
    \node[vtx] (v01) at (0,2) {};
    \node[xshift=-16pt] at (v01.west) {\scalebox{0.8}{$(0,1,0)$}};
    
    \node[novtx] (v21) at (4,2) {}; 
    \node[vtx] (v00) at (0,0) {};
    \node[xshift=0pt, yshift=-7pt] at (v00.south) {\scalebox{0.8}{$(0,0,0)$}};
    \node[vtx] (v10) at (2,0) {};
    \node[xshift=0pt, yshift=-7pt] at (v10.south) {\scalebox{0.8}{$(1,0,0)$}};
    \node[vtx] (v20) at (4,0) {}; 
    \node[xshift=0pt, yshift=-7pt] at (v20.south) {\scalebox{0.8}{$(2,0,0)$}};
     \node[vtxb] (v1/2) at (1,1) {}; 
    \node[xshift=6pt, yshift=-8pt] at (v1/2.south) {\scalebox{0.75}{$(\frac{1}{2},\frac{1}{2},h)$}};

    \draw[thick] (v00) -- (v02);
    \draw[thick] (v00) -- (v20);
    \draw[thick] (v02) -- (v20);
    \draw[thick] (v00) -- (v1/2);
    \draw[thick] (v1/2) -- (v20);
    \draw[thick] (v02) -- (v1/2);
    \draw[dashed, red] (v01) -- (v21);
    \draw[dashed, red] (v10) -- (v12);
     \draw[dashed, red] (v02) -- (v22);
    \draw[dashed, red] (v20) -- (v22);
    
    \node[vtx] (v11) at (2,2) {}; 
    \node[xshift=16pt,yshift=3pt] at (v11.north) {\scalebox{0.8}{$(1,1,0)$}};
    \node[vtxb] (v1/3) at (2,0.66) {}; 
    \node[xshift=14pt, yshift=4pt] at (v1/3.east) {\scalebox{0.75}{$(1,\frac{1}{3},\frac{2h}{3})$}};
    \node[vtxb] (v1/3') at (0.66,2) {}; 
    \node[xshift=12pt, yshift=7pt] at (v1/3'.north) {\scalebox{0.75}{$(\frac{1}{3},1,\frac{2h}{3})$}};
    
\end{tikzpicture}}
  \caption{A view of $P$ from above: the points in $\pi_x(V(Q))$ are given by the vertices of $P$, and by the intersections of edges of $P$ with the integer grid (represented by dashed lines).}
    \label{fig:pyramid}
\end{figure}

If we use Observation \ref{obs:set-covering}, the set covering problem associated with the points in $V(Q)$ that have some fractional $y_{ij}$ variable is defined by the following matrix $A_Q$:
\[\begin{array}{c|cccc}
&y_{11}&y_{12}&y_{21}&y_{22}\\
\hline
v^4&1 &  0 &1 &  0 \\
v^5&1 & 0 &1 &1 \\
v^6&1 &1 &1 &  0 \\
v^7&1 &1 &1 &1 \\
v^{11}&1 &1 &  0 &  0 \\
v^{12}& 0 &  0 &1& 1 \\
v^{13}&  0 &  0& 1 &  0 \\
v^{14}& 0  & 0 &1& 1 \\
v^{15}& 1 &  0 &  0 &  0 \\
v^{16}& 1& 1 &  0 &  0
\end{array}\]
Since the vector $z=(1,0,1,0)$ is an optimal solution to this set covering instance, we have $\lpr(Q)=2$.
This implies that the polyhedron obtained after convexifying $Q$ with respect to variables $y_{11}$ and $y_{21}$ projects down to $\conv\{x\in P: x_1,x_2\in\Z\}$.

\begin{remark}\label{rem:nonnatural}
We show that the assumption that the binarizations are natural made in Theorems \ref{thm:verticesQ} and \ref{thm:vertices} cannot be dropped.
Let $P$ be the polyhedron defined in \eqref{eq:pyramid}, and let $B_i=\{(x_i,y_{i1},y_{i2})\in \R\times [0,1]^2: x_i=y_{i1}+y_{i2},\, y_{i1}\le2y_{i2}\}$ for $i=1,2$. It is easy to check that these binarizations are not natural. Denoting by $Q'$ the corresponding binary extended formulation, we have that $\pi_x(V(Q'))$ contains, for instance, the points $(0,\frac{3}{2},0)$ and $(\frac{3}{2},0,0)$, which are not intersections of $F$ and $G_{I,\alpha}$ for any $F$, $G_{I,\alpha}$ as in Theorems \ref{thm:verticesQ} and \ref{thm:vertices}.
\end{remark}

\section{The rank of a natural binarization}\label{sec:rank}

In this section we investigate the problem of eliminating certain types of fractional vertices of our binary extended formulation through sequential convexification of the $y$-variables of a \emph{single} binarization $B$. This restriction will allow us to define a notion of rank that depends on $B$ only, provided that $B$ is natural.

 Given a polytope $Q\subseteq \R^n$,  $i\in [n]$ and $\,\alpha_1,\dots,\alpha_{\ell} \in \Z$, we 
   say that $Q$ satisfies property $\mathcal{P}^i_{\alpha_1,\dots,\alpha_{\ell}}$  if 
   every vertex $\bar{x}$ of $Q$ satisfies $\alpha_j\le \bar{x}_i$ or $\bar{x}_i\ge \alpha_j+1$ for all $j\in[\ell]$. When $\ell=1$, we call the above property $\mathcal{P}^i_{\alpha}$. Below, we consider property $\mathcal{P}^i_{\alpha_1,\dots,\alpha_{\ell}}$ of a binary extended formulation $Q$ as defined in Section \ref{sec:bin}; we stress that, although the variables of $Q$ are denoted by both $x$ and $y$, the index $i$ of the property is always relative to $x_i$.

\begin{definition}\label{def-rank} Given a  polytope $P\subseteq [0,k]^p\times \R^{n-p}$ and a binary extended formulation $Q$ of  $P$,
the {\em rank} of  $\mathcal{P}^i_{\alpha_1,\dots,\alpha_{\ell}}$ (where $i\in[p])$  is the smallest integer $t$ such that there are variables $y_{ij_1},\dots y_{ij_t}$ for which $\mathcal{P}^i_{\alpha_1,\dots,\alpha_{\ell}}$ is satisfied by $((Q_{y_{ij_1}})\dots)_{y_{ij_t}}=\conv \{(x,y)\in Q: y_{ij_1},\dots, y_{ij_t}\in \{0,1\}\})$.
The rank of  $\mathcal{P}^i_{\alpha_1,\dots,\alpha_{\ell}}$ is taken to be 0 whenever $Q$ itself satisfies the property. \end{definition}

Notice that, given a binary extended formulation $Q$ of a  polytope $P$, if $((Q_{y_{ij_1}})\dots)_{y_{ij_t}}$  satisfies property
 $\mathcal{P}^i_{\alpha_1,\dots,\alpha_{\ell}}$, then $\pi_x(((Q_{y_{ij_1}})\dots)_{y_{ij_t}})$ satisfies the property. Furthermore, by Observation \ref{obs-persistency} the property remains satisfied as long as disjunctions on the $0/1$ variables are used.

\begin{observation}\label{obs:rank}  Given a  polytope $P\subseteq [0,k]^p\times \R^{n-p}$ and a binary extended formulation $Q$ of  $P$, the rank  of $\mathcal{P}^i_{\alpha_1,\dots,\alpha_{\ell}}$ is the optimal value of the set covering instance
$\min\{ \mathbf {1}z:Az\ge \mathbf{1},\, z\in \{0,1\}^d\}$,
where the rows of $A$ are the incidence vectors of the fractional components of $\bar y_i$ for all $(\bar x,\bar y)\in V(Q)$ that violate  $\mathcal{P}^i_{\alpha_1,\dots,\alpha_{\ell}}$. 

Since $B_i$ is a binarization, if $(\bar{x},\bar{y})\in V(Q)$ satisfies $\alpha_j<\bar{x}_i<\alpha_j+1$ for some $j\in[\ell]$, then $\bar y_i\not\in\{0,1\}^d$. Therefore the above set covering instance is always feasible, as $A$ has no all-zero row. Hence the rank is well-defined.
\end{observation}

 Given  a binary extended formulation $Q$ of  $P\subseteq [0,k]^p\times \R^{n-p}$, we now study the rank of $\mathcal{P}^i_{\alpha_1,\dots,\alpha_{\ell}}$ provided the binarization $B_i$ is natural.
 The main finding is that the rank of $\mathcal{P}^i_{\alpha_1,\dots,\alpha_{\ell}}$ is the lift-and-project rank of some polytope contained in $B_i$ (Theorem \ref{thm:onebin}).
 
 In the following we write $B(f)$ to indicate $\{(x,y)\in B:x=f\}$ for any binarization $B$ and any (fractional) number $f$.

\begin{lemma}\label{le:VBif} Let $B$ be a natural binarization, $\alpha\in \{0,\dots,k-1\}$, $\alpha<f<\alpha+1$. 
Then $V(B(f))$ consists of one point $(\tilde{x},\tilde{y})$ in the relative interior of each 1-dimensional face of $B$ whose vertices $(x^u,y^u)$ and ($x^v,y^v)$ satisfy $x^u\leq \alpha$,  $x^v\geq \alpha+1$.
\end{lemma}

\begin{proof}
Since $B$ is a natural binarization, we have $\pi_x(V(B))=\{0,\dots,k\}$. Hence $\{(x,y)\in V(B):x=f\}=\emptyset$.  This shows that   $V(B(f))$ consists of one point $(\tilde{x},\tilde{y})$ in the relative interior of each 1-dimensional face of $B$ that is intersected by the hyperplane defined by $x=f$. That is, whose vertices are $(x^u,y^u)$ and ($x^v,y^v)$ with $x^u\leq \alpha$ and $x^v\geq \alpha+1$.
\end{proof}

Since, for a natural binarization $B$ and a  fractional number $f$, one has $\{(x,y)\in B(f):y\in \{0,1\}^d\}=\emptyset$, the lift-and-project rank  $\lpr(B(f))$ with respect to variables $y_1,\dots,y_d$, 
as defined in Section \ref{sec:seq-conv}, is the minimum number of $y$-variables whose convexification produces an empty polytope.

\begin{observation}\label{obs:LRBf} 
Let $B$ be a binarization, 
let $S$ be a set of 1-dimensional faces of $B$, and let $V',\,V''\subseteq B$, each consisting of at least one point in the relative interior of each face of $S$. Then $\lpr(\conv(V'))=\lpr(\conv(V''))$.
\end{observation}
\begin{proof} The set covering problems associated with $\lpr(\conv(V'))$ and $\lpr(\conv(V''))$, as defined in Observation \ref{obs:set-covering}, are the same up to duplication of identical rows of the constraint matrix. 
\end{proof}

\begin{lemma} \label{le:michi} Given a natural binarization $B$, $\alpha_1,\dots,\alpha_{\ell}\in \Z$, 
and $\alpha_j< f_j<\alpha_j+1$ for $j\in[\ell]$, let $S:=\bigcup_{j\in [\ell]}V(B(f_j))$. Then
the minimum set of 1-dimensional faces of $B$ whose relative interiors cover  $S$ coincides with 
the minimum set of 1-dimensional faces of $B$ whose relative interiors cover  $V(\conv(S))$. 
\end{lemma}
\begin{proof} Note first that by Lemma \ref{le:VBif}, $S$ is covered by the relative interiors of 1-dimensional faces of  $B$, so the statement of the lemma is meaningful. Now the lemma is implied by the following observation: 

{\em Let  $S$ be a subset of points in the relative interiors of some 1-dimensional faces of a polytope $P$.  Then $\conv(S)$ has at least one vertex (possibly two)  in each 1-dimensional face of $P$ that contains a point in $S$.}
 \end{proof}

\begin{theorem}\label{thm:onebin} Let $P\subseteq [0,k]^p\times \R^{n-p}$ be a polytope and let $Q\subseteq\R^n\times [0,1]^{pd}$ be a binary extended formulation of $P$ with binarizations $B_i\subseteq \R\times [0,1]^d$ for $i\in [p]$ that are natural. Given $\alpha_1,\dots,\alpha_{\ell}\in \Z$ and $i\in[p]$,  assume that for every $j\in [\ell]$ there is a point $(\bar{x}^j,
\bar{y}^j)\in V(Q)$ with $\alpha_j<\bar{x}^j_i<\alpha_j+1$. 
Then the rank of $\mathcal{P}^i_{\alpha_1,\dots,\alpha_{\ell}}$ equals  $\lpr(\conv(\bigcup_{j\in [\ell]} B_i(\bar{x}^j_i)))$.
\end{theorem}

\begin{proof}
Let $V':=V(\conv(\bigcup_{j\in [\ell]} B_i(\bar{x}^j_i)))$. Clearly $V'\subseteq \bigcup_{j\in [\ell]} V(B_i(\bar{x}^j_i))$. On the other hand, thanks to Lemma \ref{le:michi}, for any $j\in [\ell]$ and $(x_i,y_i)\in V(B_i(\bar{x}^j_i))$, there is a point $(x_i',y_i')\in V'$ such that, for $h\in[d]$, $y_{ih}$ is fractional if and only if $y'_{ih}$ is.

Let $A_1$ be the constraint matrix of the set covering problem relative to $\mathcal{P}^i_{\alpha_1,\dots,\alpha_{\ell}}$, as defined in Observation \ref{obs:rank}. Further, let $A_2$ be the constraint matrix of the set covering problem relative to $\lpr(V')$ (with respect to variables $y_{i1},\dots,y_{id}$), as defined in Observation \ref{obs:set-covering}. We will show that $A_1$ and $A_2$ have the same set of rows (i.e., they are the same matrix up to repeating identical rows), concluding the proof.

First, let $a\in \{0,1\}^d$ be a row of $A_1$. Then there is a vertex $(\tilde{x},\tilde{y})\in V(Q)$ that violates $\mathcal{P}^i_{\alpha_1,\dots,\alpha_{\ell}}$, such that the fractional coordinates of $\tilde{y}_i$ correspond to the 1-entries of $a$. Let $j\in [\ell]$ satisfy $\alpha_j<\tilde{x}_i<\alpha_j+1$. Thanks to Theorem \ref{thm:verticesQ} $(\tilde{x}_i,\tilde{y}_i)$ is a vertex of $B_i(\tilde{x}_i)$, in particular it lies in the relative interior of a 1-dimensional face of $B_i$. But then there is a vertex of $B_i(\bar{x}_i^j)$ (hence, there is a vertex of $V'$) that lies in the same 1-dimensional face, implying that  $a$ is also a row of $A_2$.

On the other hand, let $a\in \{0,1\}^d$ be a row of $A_2$. Then there are $j\in[\ell]$ and a vertex $(\bar{x}^j_i,\hat{y}_i)$ of $B_i(\bar{x}_i^j)$ such that the fractional coordinates of $\hat{y}_i$ correspond to the ones of $a$.
One can check, using Theorem \ref{thm:verticesQ} and the point $(\bar{x}^j,
\bar{y}^j)\in V(Q)$ given by the hypotheses, that there is a vertex $(\bar{x}^j, \tilde{y})\in V(Q)$ such that $\tilde{y}_i=\hat{y}_i$. This shows that $a$ is a row of $A_1$ as well.
\end{proof}

We remark that, under the assumptions of the above theorem, the rank of $\mathcal{P}^i_{\alpha_1,\dots,\alpha_{\ell}}$ only depends on $B_i$ and $\alpha_1,\dots, \alpha_\ell$, and not on $P$ or on the other binarizations of $Q$. We now give another description of the rank of $\mathcal{P}^i_{\alpha_1,\dots,\alpha_{\ell}}$ that will be useful in the following.

 Recall that the skeleton $\sk(P)$ of a polytope $P$ is the simple graph whose vertices are the vertices of $P$, where two vertices are adjacent in $\sk(P)$ if and only if they are adjacent  in $P$. 
 Given the skeleton $\sk(B)$ of a  natural binarization $B\subseteq [0,k]\times [0,1]^d$ and $\alpha\in \{0,\dots,k-1\}$, we say that edge $((x^u,y^u),(x^v,y^v))$ is an {\em $\alpha$-edge} if $x^u\le \alpha$ and $x^v\ge \alpha+1$, or viceversa.
 The {\em indicator vector} of edge $((x^u,y^u),(x^v,y^v))$ is the  vector  $t\in \{0,1\}^d$  where $t_k=0$ if $y^u_k$ and $y^v_k$ are either both 0 or both 1.

\begin{definition}\label{def:rankB}
  Given a natural binarization $B\subseteq [0,k]\times [0,1]^d$,   $\alpha_1,\dots,\alpha_\ell\in \{0,\dots,k-1\}$, and   $\bar{x}^j$  such $\alpha_j<\bar{x}^j<\alpha_j+1$ for $j\in[\ell]$, we define the {\em rank} of $B$ with respect to $\alpha_1,\dots,\alpha_\ell$ as follows:
  $$\rk_B(\alpha_1,\dots,\alpha_\ell):=\lpr(\conv(\textstyle\bigcup_{j\in [\ell]} B(\bar{x}^j))).$$
\end{definition}

 \begin{observation}\label{obs:skeleton} 
 Given a natural binarization $B\subseteq [0,k]\times [0,1]^d$ and $\alpha_1,\dots,\alpha_\ell\in \{0,\dots,k-1\}$, $\rk_B(\alpha_1,\dots,\alpha_\ell)$ is the rank of $\mathcal{P}^i_{\alpha_1,\dots,\alpha_{\ell}}$ with respect to any binary extended formulation satisfying the hypotheses of Theorem \ref{thm:onebin} whose associated $i$-th binarization is $B$.
 
Furthermore $\rk_B(\alpha_1,\dots,\alpha_\ell)$  is the optimum value of the set covering problem in which the rows of the constraint matrix are the indicator vectors of the $\alpha_j$-edges,  $j\in[\ell]$, of $\sk(B)$. 
 \end{observation}

 In the next subsections we will investigate the parameter $\rk_B$ for different natural binarizations introduced in the literature. 



We recall a result of Dash, G\"unluk and Hildebrand \cite{dash2018binary} that highlights the importance of property $\mathcal{P}_{\alpha}$ with respect to general splits involving variables of a single binarization.

\begin{theorem}\label{thm:dash} (Proposition 6 in \cite{dash2018binary}) Given a binary extended formulation $Q$ of a polytope $P\subseteq [0,k]^p\times \R^{n-p}$ with binarizations $B_1,\dots,B_p\subseteq \R\times [0,1]^d$ (not necessarily natural), let $\pi\in \Z^d$, $\pi_0\in \Z$ and $i\in [p]$.
Let $Q^0:=\{(x,y)\in Q:\pi y_i\le \pi_0\}$, $Q^1:=\{(x,y)\in Q:\pi y_i\ge \pi_0+1\}$. Then there exists  $\alpha\in \{0,\dots,k-1\}$ such that
$$\conv(\{(x,y)\in Q:x_i\le \alpha\})\cup
\{(x,y)\in Q:x_i\ge \alpha+1\})\subseteq \conv (Q^0\cup Q^1).$$ 
\end{theorem}

(The original statement given in \cite{dash2018binary} is weaker. However their argument proves the above theorem.)

This theorem indicates that split disjunctions in the extended space that involve $y$-variables of a single binarization are weaker than (canonical) split disjunctions in the $x$-space. On the other hand, the authors of \cite{dash2018binary} argue that a split disjunction that involves $y$-variables belonging to distinct binarizations may be more powerful than any split disjunction in the $x$-space.


% ... 

\subsection{The rank of the unary and full binarizations}


  We now examine
  $\rk_B(\alpha_1,\dots,\alpha_{\ell})$ for the unary and the full binarizations, as defined at the end of Section \ref{sec:properties}.
  \smallskip

Binarizations $B^U(d)$ and $B^F(d)$ are $d$-dimensional polytopes with  $d+1$ vertices. Therefore they are simplices and the skeleton of these binarizations is a complete graph on $d+1$ vertices.

\begin{observation}\label{obs:unary}
 The vertices of $B^U(d)$ are $(i,y(i))$ for $i=0,\dots,d$, where $y(i)\in \{0,1\}^d$ has the first $i$ components equal to 1 and the others to 0. The indicator vector $t$ of edge  $((i,y(i)), (j,y(j)))$, where $i<j$, has $t_k=1$ for  $i+1\le k\le j$ and  $t_k=0$ otherwise.  Thus any  matrix whose rows are indicator vectors of edges of the skeleton of $B^U(d)$  has the consecutive ones property and is therefore totally unimodular. 
\end{observation}

\begin{lemma}\label{le:rank-unary} Given pairwise distinct  $\alpha_1,\dots,\alpha_{\ell}\in \{0,\dots,d-1\}$, we have that \[\rk_{B^U(d)}(\alpha_1,\dots,\alpha_{\ell})=\ell.\]
\end{lemma}
\begin{proof}
By Observations \ref{obs:skeleton} and \ref{obs:unary}, the vector $z\in\{0,1\}^d$ defined by $z_{\alpha_1+1}=\dots=z_{\alpha_{\ell}+1}=1$ and $z_k=0$ otherwise is an optimal solution to the set covering problem associated with $\rk_{B^U(d)}(\alpha_1,\dots,\alpha_{\ell})$.
\end{proof}

We now turn our attention to the full binarization.
For $i\in[d]$, we denote by $\mathbf e_i$ the canonical vector with the $i$-th component equal to one.

 \begin{observation}\label{obs:full}
 The vertices of  $B^F(d)$ are  $(0,0)$ and $(i,\mathbf{e}_i)$ for all $i\in[d]$. 
  The indicator vector of edge  $((0,0),(i,\mathbf{e}_i))$ is $\mathbf{e}_i$. 
  The indicator vector of edge  $((i,\mathbf{e}_i),(j,\mathbf{e}_j))$, where $i\ne j$, is $\mathbf{e}_i+\mathbf{e}_j$.
  \end{observation}

\begin{lemma}\label{le:rank-full} Given $\alpha_1,\dots,\alpha_{\ell}\in \{0,\dots,d-1\}$, 
we have that 
\[\rk_{B^F(d)}(\alpha_1,\dots,\alpha_{\ell})=d-\min_{j\in[\ell]}\alpha_j.\]
\end{lemma}
\begin{proof}  Given $\alpha \in \{0,\dots,d-1\}$, the $\alpha$-edges are $((0,0), (\gamma,\mathbf{e}_{\gamma}))$ and $((\beta,\mathbf{e}_{\beta}), (\gamma,\mathbf{e}_\gamma))$ where $\beta \le \alpha $ and $\gamma \ge \alpha+1$. 
  Hence, by Observation \ref{obs:skeleton}, the optimal solution to the set covering problem associated with $\rk_{B^F(d)}(\alpha_1,\dots,\alpha_{\ell})$ is the vector $z$ defined by $z_k=1$ for $k\ge \min_{j\in[\ell]}\alpha_j +1$, $z_k=0$ otherwise.  
\end{proof}

Note that Lemmas \ref{le:rank-unary} and \ref{le:rank-full} show that $\rk_{B^F(d)}(\alpha_1,\dots,\alpha_{\ell})\geq\rk_{B^U(d)}(\alpha_1,\dots,\alpha_{\ell})$  for any choice of pairwise distinct  $\alpha_1,\dots,\alpha_{\ell}\in \{0,\dots,d-1\}$. Therefore the unary binarization is preferable, according to this criterion.

\subsection{Affine binarizations} 

 The constraints of the set covering problem related to the rank of a natural binarization $B$, as defined in  Observation \ref{obs:skeleton}, are associated with the edges of the skeleton of $B$. 
We remark that the skeleton of $B$ is in general different from the skeleton of its projection $\pi_y(B)$.

\begin{observation}\label{obs:polytopes}

Let $B$ be a binarization with variables $(x,y)\in [0,k]\times [0,1]^d$.
\begin{enumerate}

\item If $|V(\pi_y(B))|=|V(B)|$ then $\pi_y$ is a bijection between $V(B)$ and $V(\pi_y(B))$.

\item If $|V(\pi_y(B))|=|V(B)|$  and $V'$ is the set of vertices of a face of $\pi_y(B)$, then the preimage of $V'$ under $\pi_y$ is the set of vertices of a 
face of $B$.

\item If $B$ is integral, then $|V(B)|=|V(\pi_y(B))|$.

\item If $B$ is affine, then $B$ and $\pi_y(B)$ are isomorphic.
\end{enumerate}
\end{observation}
\begin{proof} 
1. trivially holds.

Let $F=\conv(V')$ be a face of  $\pi_y(B)$, where $V'\subseteq V(\pi_y(B))$, and let $ay\le \beta$ be an inequality that is valid for $\pi_y(B)$ and exposes $F$. Then $ay\le \beta$ is valid for $B$ and exposes precisely the preimage of  $F=\conv(V')$, which is therefore a face of $B$. By 1., the preimage of $V'$ is a subset of $V(B)$, and this proves 2.

If $(x,y)$ is a vertex of $B$ and $B$ is an integral binarization, then $y\in\{0,1\}^d$. Therefore $y$ is also a vertex of $\pi_y(B)$, as this set is contained in $[0,1]^d$. This proves 3.

If $B$ is an affine binarization, there is an equation of the form $x=ay+\beta$  satisfied by all points in $B$. This immediately implies 4.
\end{proof}

\begin{corollary}\label{cor:subgraph}
If $B$ is an integral binarization, then $\sk(\pi_y(B))$  is a subgraph of $\sk(B)$ on the same vertex set. If $B$ is an affine binarization, then  $\sk(B)$ and $\sk(\pi_y(B))$  are isomorphic.  
\end{corollary}

We say that a binarization $B\subseteq \R\times [0,1]^d$ is a $d$-hypercube binarization, or simply a hypercube binarization, if $\pi_y(V(B))=\{0,1\}^d$ and $\pi_x(V(B))=\{0,\dots,2^d-1\}$, i.e., if the vertices of $B$ are in one-to-one correspondence with the vertices of the $d$-hypercube, and each has a different $x$-component. We will use the fact that a hypercube binarization is perfect (hence natural, exact and integral, see Observation \ref{obs:twox}).

The logarithmic binarization $B^L(d)$ is a $d$-hypercube binarization with vertices $(x, y)$, where $x\in\{0,\dots,2^d-1\}$ and $y=(x)_2$ is the reverse of the vector in $\{0,1\}^d$ expressing $x$ in base 2. 

In the following, we will write $B^L$ instead of $B^L(d)$ when the dimension $d$ is clear from the context.

\begin{lemma}\label{le:onlyaffine}
Up to permuting and complementing variables, the logarithmic binarization  is the only hypercube binarization that is affine.
\end{lemma}

\begin{proof}
Let $B$ be an affine hypercube binarization. Then there exist $a\in\R^d$ and $\beta\in\R$ such that $x=ay+\beta$ for every $(x,y)\in B$. We first show that we may assume that $\beta=0$. Indeed, consider the vertex $(0,y^0)$ of $B$. By complementing the variables such that $y^0_i=1$, we obtain $y^0=0$, hence $\beta=0$. 

Since $B$ is affine, by Observation \ref{obs:polytopes}, $\sk(B)$ is the $d$-hypercube. Vertices $(a_i,\mathbf{e}_i)$ show that $a_1,\dots,a_d \ge1$. Using  $\pi_x(V(B))=\{0,\dots,2^d-1\}$, we now verify that  $\{a_1,\dots,a_d\}=\{2^0,\dots,2^{d-1}\}$.

Let $y^i$ be such that $(i,y^i)\in B$. Notice that, for any $j$ such that $y^i_j=1$, we have $0\leq a(y^i-e_j)=i-a_j\le i-1$. Now, we show that $y^{2^h}\in \{\mathbf{e}_1,\dots, \mathbf{e}_d\}$ for $h=0,\dots, d-1$ by induction on $h$. For $h=0$, the claim is true as otherwise there are at least two $j$'s with $y^1_j=1$, but only one value for $a(y^1-\mathbf{e}_j)=0$. Now, let $h>0$. By induction, and without loss of generality, we can assume that $y^{2^0}=\mathbf{e}_1$, $\dots$, $y^{2^{h-1}}=\mathbf{e}_{h}$,  implying that $y^i=(i)_2$ for $i\leq 2^{h}-1$. Hence, we must have $y^{2^h}_j=1$  for at least one $j>h$. But, on the other hand, if there are two such indices $j$,  then we reach a contradiction as before, as there would be two values $i', i''$ smaller than $2^h$ but different from $0,\dots, 2^h-1$; we argue similarly if there is another index $j'\leq h$  with $y^i_{j'}=1$.
\end{proof}

\subsection{The rank of hypercube binarizations} \label{sec:hypercube}

In this section we  show that, among hypercube binarizations, the logarithmic binarization $B^L$ minimizes $\rk_B(\alpha_1,\dots,\alpha_\ell)$ for any choice of $\alpha_1,\dots,\alpha_\ell$.
 Given a hypercube binarization $B$, we let $G(B)$ be the graph with vertex set $V(B)$, where $(x_i,y_i)$, $(x_j,y_j)$ in $V(B)$ are adjacent if and only if the 0/1 vectors $y_i$ and $y_j$ differ in exactly one component. Hence, $G(B)$ is isomorphic to the $d$-dimensional hypercube graph.
 
By Observation \ref{obs:polytopes} and Corollary \ref{cor:subgraph}, we have that $G(B)=\sk(\pi_y(B))$ and $G(B)$ is a subgraph of $\sk(B)$ having the same vertex set. If the binarization is a logarithmic binarization $B^L$, then  $G(B^L)=\sk(B^L)$ because $\sk(\pi_y(B^L))=\sk(B^L)$.

For $\alpha\in \{0,\dots,2^d-2\}$ we let $V_{\alpha}:=\{(x,y)\in V(B):x\le \alpha\}$, and for $i\in [d]$ we define $V_0^i:=\{(x,y)\in V(B): y_i=0\}$ and $V_1^i:=\{(x,y)\in V(B): y_i=1\}$. We define an edge $e$ of $G(B)$ to be of {\em type $i$} if $e\in \delta(V_0^i)$, and we write $\type(e)=i$.

\begin{observation}\label{obs:hypercube}
Given a $d$-hypercube binarization $B$ and  its hypercube graph $G(B)$: \begin{enumerate}
    \item For $\alpha\in \{0,\dots,2^d-2\}$, the cut $\delta(V_{\alpha})$ consists of the $\alpha$-edges of $G(B)$.
\item For $i\in [d]$, the cut $\delta(V^i_0)$ consists of the edges of $G(B)$ whose indicator vector is $\mathbf{e}_i$. \item The cuts $\delta(V^i_0)$, $i\in [d]$, partition the edges of $G(B)$ into $d$ perfect matchings.
\end{enumerate}
\end{observation}

We now argue that, given a hypercube binarization $B$ and $\alpha_1,\dots,\alpha_\ell\in\{0,\dots, 2^d-2\}$, in order to compute $\rk_B(\alpha_1,\dots,\alpha_\ell)$ one can restrict to the set covering problem related to $G(B)$ instead of $\sk(B)$ (as in Observation \ref{obs:skeleton}).   
This allows us to draw conclusions on general (not necessarily affine) hypercube binarizations.

\begin{lemma}\label{lem:logbincube}  
Let $B$ be a $d$-hypercube binarization, and $G(B)$ be the corresponding graph. For $\alpha_1,\dots,\alpha_\ell\in\{0,\dots, 2^d-2\}$, we have
\[\rk_B(\alpha_1,\dots,\alpha_\ell) = \big|\{i:\delta(V_0^i)\cap\textstyle\bigcup_{j=1}^\ell\delta(V_{\alpha_j})\ne \emptyset\}\big|.
\]
\end{lemma}
\begin{proof}
Let $S:=\{i:\,\delta(V_0^i)\cap\textstyle\bigcup_{j=1}^\ell\delta (V_{\alpha_j})=\emptyset\}$.
We will show that the  incidence vector of $S$ is an optimal solution of the set covering problem relative to $\rk_B(\alpha_1,\dots,\alpha_\ell)$ (see Observation \ref{obs:skeleton}).

First, for $j\in [\ell]$, consider $e\in \delta(V_{\alpha_j})$: it is an $\alpha_j$-edge of $\sk(B)$  whose indicator vector is $\mathbf{e}_{\type(e)}$. 
This shows that  $S$ is contained in any solution of the aforementioned set covering problem. 

On the other hand, we now show that $S$ is a feasible solution of the set covering problem, concluding the proof. In particular we show that, for any fixed $j\in [\ell]$ and any $\alpha_j$-edge of $G(B)$ with extreme points $(x^1,y^1)$, $(x^2,y^2)$ and indicator vector $t$,
 there is an edge $e\in \delta(V_{\alpha_j})$ such that $t_{\type(e)}$=1.
 
 Consider any shortest path $\rho$ in $G(B)$ joining nodes $(x^1,y^1)$, $(x^2,y^2)$: it contains exactly one edge of type $i$ for each nonzero coordinate of $t$. Since we must have, without loss of generality, $x^1\leq \alpha_j$ and $x^2\geq \alpha_j+1$, there is an edge $e$ of $\rho$ with $e\in \delta(V_{\alpha_j})$. 
\end{proof}

Given a hypercube binarization $B$, its hypercube graph $G(B)$ and $\alpha_1,\dots,\alpha_{\ell}\in\{0,\dots,2^d-2\}$, we let $f_B(\alpha_1,\dots,\alpha_\ell):=|\{i:\delta(V_0^i)\cap\textstyle\bigcup_{j=1}^\ell\delta(V_{\alpha_j})=\emptyset\}|$.
As   $|\{i:\delta(V_0^i)\cap\textstyle\bigcup_{j=1}^\ell\delta(V_{(\alpha_j})\ne \emptyset\}|=d-f_B(\alpha_1,\dots,\alpha_\ell)$, by Lemma \ref{lem:logbincube} we have the following:

\begin{corollary}\label{cor:rankcube}  
Let $B$ be a $d$-hypercube binarization, and $G(B)$ be the corresponding graph. For $\alpha_1,\dots,\alpha_\ell\in\{0,\dots, 2^d-2\}$, we have
\[\rk_B(\alpha_1,\dots,\alpha_\ell) = d-f_B(\alpha_1,\dots,\alpha_\ell).
\]
\end{corollary}

\begin{lemma}\label{lem:matchings}
Given a hypercube binarization $B$, its hypercube graph $G(B)$  and $\alpha_1,\dots,\alpha_\ell\in\{0,\dots, 2^d-2\}$, we have that $2^{f_B(\alpha_1,\dots,\alpha_\ell)}$ divides $\alpha_j+1$ for every $j\in [\ell]$.
\end{lemma}
\begin{proof}  Given a subset of vertices $U$ of $G(B)$, we let $G[U]$ be the subgraph of $G(B)$ induced by $U$. Define $q:=f_B(\alpha_1,\dots,\alpha_\ell)$ and let $\delta(V_0^{i_1}),\dots,\delta(V_0^{i_q})$ be the cuts (and perfect matchings) that do not intersect any $\delta (V_{\alpha_j})$ for $j\in [\ell]$.   Then   for $j\in [\ell]$  and $k\in [q]$ the restriction of  $\delta(V_0^{i_k})$ to $G[V_{\alpha_j}]$ is a cut and a perfect matching of $G[V_{\alpha_j}]$.

We 
proceed by induction on $q$.   
For $q=1$ the lemma holds as, if $G[V_{\alpha_j}]$ admits a perfect matching, then $|V_{\alpha_j}|$ is even.

Assume  $q>1$ and let $V_{\alpha_j}^0=V_{\alpha_j}\cap V^{i_q}_0$ and $V_{\alpha_j}^1=V_{\alpha_j}\cap V^{i_q}_1$: this is a partition of $V_{\alpha_j}$ into equal parts, and since  $\delta(V_{\alpha_j}^0)$ is a cut of $G[V_{\alpha_j}]$, we have that edges of type $i_1,\dots, i_{q-1}$ induce a perfect matching in $G[V_{\alpha_j}^0]$. Then by induction $2^{q-1}$ divides $|V_{\alpha_j}^0|=\frac{|V_{\alpha_j}|}{2}$, and thus $2^q$ divides $|V_{\alpha_j}|=\alpha_j+1$.
\end{proof}

\begin{lemma}\label{lem:logarithmicMultiple}
Given the logarithmic binarization $B^L$  and $\alpha\in\{0,\dots, 2^d-2\}$, we have that $f_{B^L}(\alpha)$ is the largest number of 1's that are not preceded by any 0 in $(\alpha)_2$. Equivalently, $f_{B^L}(\alpha)$ is the largest $t$ such that $2^t$ divides $\alpha+1$.
\end{lemma}
\begin{proof} 
We prove the first statement, as the second easily follows from the first.
 Consider the $i$-th bit of $(\alpha)_2$. If it is 0,  vertices $(\alpha,(\alpha)_2)$ and $(\alpha+2^{i-1},(\alpha+2^{i-1})_2)$ are adjacent in $G(B^{L})$ and the  corresponding  edge is in $\delta(V^i_0)\cap \delta(V_{\alpha})$.
 
 If the $i$-th bit of $(\alpha)_2$ is 1 and it is preceded by a 0 in position  $j<i$, then vertices $(\alpha - 2^{i-1}+2^{j-1},(\alpha - 2^{i-1}+2^{j-1})_2)$ and $(\alpha +2^{j-1},(\alpha +2^{j-1})_2)$ are adjacent in $G(B^{L})$ and the   corresponding edge is  in $\delta(V^i_0)\cap \delta(V_{\alpha})$.
 
 Finally, if the $i$-th bit of $(\alpha)_2$ is 1 and it is not preceded by a 0, we claim that no $\alpha$-edge has type $i$, because if $\alpha'\le\alpha$ and  $\alpha''\ge\alpha+1$ have $(\alpha')_2$ and  $(\alpha'')_2$ different in a single position, that position cannot be the $i$-th. 
 \end{proof}

 The following corollary, together with Corollary \ref{cor:rankcube}, gives an expression for the rank of the logarithmic binarization.

\begin{corollary}\label{cor:logrank}
Given the logarithmic binarization $B^L$, its hypercube graph $G(B^L)$  and $\alpha_1,\dots,\alpha_\ell\in\{0,\dots, 2^d-2\}$, we have that $f_{B^L}(\alpha_1,\dots,\alpha_\ell)=\max \{t:2^t\mbox{ divides }\alpha_j+1\:\forall j\in [\ell]\}$.
\end{corollary}



We are now ready for the main result of this subsection.
\begin{theorem}\label{thm:logbest}
Let $B$ be a $d$-hypercube binarization, $B^L$ be the logarithmic binarization and let $\alpha_1,\dots,\alpha_\ell\in\{0,\dots,2^d-2\}$. Then \[
\rk_B(\alpha_1,\dots,\alpha_\ell) \geq \rk_{B^{L}}(\alpha_1,\dots,\alpha_\ell).
\]
\end{theorem}
\begin{proof} By Corollary \ref{cor:rankcube},
$\rk_B(\alpha_1,\dots,\alpha_\ell) = d-f_B(\alpha_1,\dots,\alpha_\ell)$ and $\rk_{B^L}(\alpha_1,\dots,\alpha_\ell) = d-f_{B^L}(\alpha_1,\dots,\alpha_\ell)$.
By Lemma  \ref{lem:matchings} and Corollary \ref{cor:logrank},  
$f_{B^L}(\alpha_1,\dots,\alpha_\ell)\ge f_{B}(\alpha_1,\dots,\alpha_\ell)$.
\end{proof}

We conclude this subsection by comparing our result with Theorem 11 in \cite{owen2002value}: Informally, that theorem states that, given a binary extended formulation $Q$ and a variable $x$ whose binarization  is $B=B^L(d)$, convexifying all the $y$-variables of $B$ but one is not enough to ensure integrality of $x$. A stronger result can be derived from our Corollary \ref{cor:logrank}, as $\rk_{B^L}(\alpha)=d$ for any even $\alpha$. Thanks to Theorem \ref{thm:logbest}, one can extend this to all hypercube binarizations. 


\subsection{Truncated hypercube binarizations}\label{sec:trunc}

For $d\geq 1$ and $2^{d-1}< v\leq 2^d$, the \emph{truncated} logarithmic binarization $B^L_{< v}(d)$ is the convex hull of points $(x,(x)_2)$ for $0\leq x\leq v-1$ (where again $(x)_2$ is the reversed $d$-bit string expressing $x$ in base 2). Note that the assumption $2^{d-1}< v\leq 2^d$ is taken without loss of generality, and, if $v=2^d$, then $B^L_{< v}(d)=B^L(d)$. The projection of $B^L_{< v}(d)$ on the $y$-space is called a $d$-dimensional {\em truncated hypercube} and is also known as a rev-lex polytope.  As the latter is the convex hull of integral vectors ordered lexicographically, it belongs to a well-studied class of polytopes, see e.g. \cite{gillmann2006revlex, gupte2016convex,conforti2020scanning}. In particular a description of $B^L_{< v}(d)$ is obtained by adding to $B^L(d)$ at most $d$ inequalities.  
We start with a couple of preliminary observations; the first follows from Corollary \ref{cor:subgraph}, the second is an easy check and the third is proved in \cite{gillmann2006revlex}.

\begin{observation}\label{obs:trunc}
Let $B^L_{< v}(d)$ be as defined above.
\begin{enumerate}
    \item Graphs $\sk(B^L_{< v}(d))$ and $\sk(\pi_y(B^L_{< v}(d)))$ are isomorphic.
    \item 
    $B^L_{< v}(d)$ is the convex hull of $B^L(d-1)$ and $(2^{d-1},\mathbf e_d)+B^L_{< v'}(d')$, where $v'=v-2^{d-1}$ and $d'=\lceil{\log(v')}\rceil$. (Here we are abusing notation and embedding $B^L(d-1)$ and $B^L_{< v'}(d')$ in $\R\times[0,1]^d$.)
    \item If two vertices of $B^L_{< v}(d)$ are adjacent vertices of $B^L(d)$, then they are adjacent in $B^L_{< v}(d)$ as well.
\end{enumerate}
\end{observation}
 
\begin{lemma}\label{lem:trunc}
For any $2^{d-1}< v\leq 2^d$ and $\alpha\in \{0,\dots,v-2\}$, we have:
\begin{itemize}
    \item if $\alpha< 2^{d-1}$, then $\rk_{B^L_{< v}(d)}(\alpha)=\rk_{B^L(d)}(\alpha)$;
    \item if $\alpha\geq 2^{d-1}$, then $\rk_{B^L_{< v}(d)}(\alpha)=1+\rk_{B^L_{< v'}(d')}(\alpha-2^{d-1})$, where $v'=v-2^{d-1}$ and $d'=\lceil{\log(v')}\rceil$.
\end{itemize}
\end{lemma}
\begin{proof}
Fix $\alpha \in \{0,\dots,v-2\}$, and let $I_L(\alpha,d)$ (resp. $I_{<v}(\alpha,d)$) be the support of an optimal solution of the set covering problem related to $\rk_{B^L(d)}(\alpha)$ (resp. $\rk_{B^L_{< v}(d)}(\alpha)$) as in Observation \ref{obs:skeleton}. 

Let $U_i=\{(x,y)\in \R\times \R^d: y_d=i\}$ for $i=0,1$.
Notice that $U_0, U_1$ induce a partition of the vertices of $B^L(d)$ and of $B^L_{<v}(d)$: in particular, the points of $U_0$ that are vertices of $B^L(d)$ and of $B^L_{<v}(d)$ are exactly the same, while the points of $U_1$ that are vertices of $B^L_{<v}(d)$ are also vertices of $B^L(d)$. 

\begin{itemize}
    \item Let $0\leq \alpha<2^{d-1}$. We show that $I_{<v}(\alpha,d)=I_L(\alpha,d)$. 
    First, notice that all the $\alpha$-edges of $\sk(B^L(d))$ are between $U_0$ and $U_1$, or between two nodes in $U_0$, and that this is true for the $\alpha$-edges of $\sk(B^L_{<v}(d))$, as well. 
    \begin{itemize}
    \item We first deal with the $\alpha$-edges between $U_0$ and $U_1$. We claim that $d\in I_L(\alpha,d)\cap I_{<v}(\alpha,d)$. Indeed, we have that points $(0,\dots, 0)$ and $(2^{d-1},0,\dots,0,1)$ are adjacent vertices of both $B^L(d)$ and $B^L_{<v}(d)$ (thanks to part 3 of Observation \ref{obs:trunc}). The corresponding edge is an $\alpha$-edge whose indicator vector is $\mathbf{e}_{d}$, proving our claim.
    \item Now, we focus on the $\alpha$-edges between pairs of vertices in $U_0$. Since,  by part 2 of Observation \ref{obs:trunc}, $U_0$ induces the same graph on both $\sk(B^L(d))$ and $\sk(B^L_{<v}(d))$,
     \[
    I_{<v}(\alpha,d)\setminus\{d\}= I_{L}(\alpha,d)\setminus\{d\}.\] 
    This concludes the proof of the first case.
  \end{itemize}  
    \item Now, let $2^{d-1}\leq \alpha \leq v-2$. Then, all the $\alpha$-edges of $\sk(B^L(d))$ are between $U_0$ and $U_1$, or between two nodes in $U_1$, and that this is true for the $\alpha$-edges of $\sk(B^L_{<v}(d))$, as well. 
\begin{itemize}
    \item We claim that $d\in I_{<v}(\alpha,d)$: indeed, consider $v_i=(x^i,(x^i)_2)\in$ $B^L_{<v}(d)$ for $i=1,2$ with $x^1=\alpha+1$ and $x^2=\alpha+1-2^{d-1}$. Then $v_1, v_2$ are adjacent in $B^L_{<v}(d)$ (thanks to part 3 of Observation \ref{obs:trunc}) and the corresponding edge is an $\alpha$-edge whose indicator vector is $\mathbf{e}_{d}$.
    \item Now, we only need to deal with $\alpha$-edges between pair of vertices in $U_1$. Thanks to part 2 of Observation \ref{obs:trunc}, we have $I_{<v}(\alpha,d)\setminus\{d\}= I_{<v'}(\alpha-2^{d-1},d')$, which concludes the proof.
\end{itemize}
\end{itemize}


\end{proof}

The following gives a direct way to compute the rank of $B^L_{< v}(d)$. We restrict to $v<2^d$, as $B^L_{<2^d}(d)=B^L(d)$.

\begin{corollary}\label{cor:trunc}
For any $2^{d-1}< v< 2^d$ and $\alpha\in \{0,\dots,v-2\}$, let $j\in[d]$ be the largest index in which $(v)_2$ is 1 and $(\alpha)_2$ is 0 (notice that there must be at least one) and let $s$ be the number of bits $j'>j$ where both $(v)_2$ and $(\alpha)_2$ are 1. Let $\tilde \alpha$ be the number such that $(\tilde{\alpha})_2$ is obtained from $(\alpha)_2$ by removing all bits $j'>j$, and let $\tilde{d}$ be the number of bits of $(\tilde{\alpha})_2$. Then:
\[
\rk_{B^L_{< v}(d)}(\alpha) = s+\rk_{B^L(\tilde{d})}(\tilde{\alpha}).
\]
\end{corollary}

\begin{proof}
If $\alpha< 2^{d-1}$ then $j$ as defined in the statement is $d$, hence $\tilde\alpha=\alpha$, $\tilde d=d$, $s=0$. The result then follows from the first part of Lemma \ref{lem:trunc}.

If $\alpha\ge2^{d-1}$, we have $j<d$. Let $j_1,\dots, j_s$ be the indices $j'>j$ where both $(v)_2$ and $(\alpha)_2$ are 1, with $j=:j_0<j_1<\dots<j_s=d$. We remark that for each index $j'$ with $j_i<j'<j_{i+1}$ for some $i\in\{0,\dots, s-1\}$, we have that bit $j'$ of both $(v)_2$ and $(\alpha)_2$ is 0.

For $i\in\{0,\dots, s\}$, we let $v_i$ be the number such that $(v_i)_2$ is obtained from $(v)_2$ by removing the bits $j'>j_i$, and we define $\alpha_i$ similarly: note that $v_s=v$, $\alpha_s=\alpha$, and $\alpha_0=\tilde \alpha$ as defined in the statement. Furthermore,  $v_i=v_{i+1}-2^{j_{i+1}-1}$, $\alpha_i=\alpha_{i+1}-2^{j_{i+1}-1}$  for $i\leq s-1$. Moreover, let $d_i=\lceil \log v_i \rceil$ for $i\in\{0,\dots, s\}$: again we have $d_s=d$ and $d^0=\tilde d$ as defined in the statement.

Hence, applying the second part of Lemma \ref{lem:trunc} $s$ times, we get:
\[\rk_{B^L_{< v}(d)}(\alpha)=1+\rk_{B^L_{< v_{s-1}}(d_{s-1})}(\alpha_{s-1})=\dots\]
\[
=s + \rk_{B^L_{< v_{0}}(d_{0})}(\alpha_{0})= s+ \rk_{B^L(\tilde{d})}(\tilde{\alpha}),
\]
where the last equation follows from the first part of Lemma \ref{lem:trunc}.
\end{proof}

\begin{corollary}\label{cor:trunc2}
For any $2^{d-1}< v\leq 2^d$ and $\alpha\in \{0,\dots,v-2\}$, we have that \[\rk_{B^L_{< v}(d)}(\alpha)\leq \rk_{B^L(d)}(\alpha).\]
\end{corollary}
\begin{proof}
We use the same notation as in Corollary \ref{cor:trunc}, which states that $\rk_{B^L_{< v}(d)}(\alpha) = s+\rk_{B^L(\tilde{d})}(\tilde{\alpha})$. Consider the $d$-bit string $(\alpha)_2$ and the $\tilde{d}$-bit string $(\tilde \alpha)_2$ obtained from the former by removing all bits $j'>j$: since the $j$-th bit of $(\alpha)_2$ is 0, the two strings have the same number of 1's that are not preceded by any zero, i.e. by Lemma \ref{lem:logarithmicMultiple}, $f_{B^L(\tilde{d})}(\tilde{\alpha})=f_{B^L(d)}(\alpha)$. Moreover, we remark that $s\leq d-\tilde{d}$. Applying again Lemma \ref{lem:logarithmicMultiple} and Corollary \ref{cor:rankcube} we conclude that 
 \[\rk_{B^L_{< v}(d)}(\alpha)\leq d-\tilde{d}+ \rk_{B^L(\tilde{d})}(\tilde{\alpha})
 =d-f_{B^L(d)}(\alpha)=\rk_{B^L(d)}(\alpha).
 \]
\end{proof}

In light of the previous results, it is natural to ask whether the truncated logarithmic binarization is ``best possible'' in terms of rank among other binarizations with the same projection on the $y$-space, similarly as the complete logarithmic binarization is among all hypercube binarizations: however, the following example shows that the answer is in general negative. 
Consider the binarization $B:=\{(x,y_1,y_2)\in \R\times [0,1]^2: x=-y_1+y_2+1, y_1+y_2\leq 1\}$. Notice that $B=\conv\{(1,0,0),(0,1,0),(2,0,1)\}$. We have $\rk_B(0)=1$ (as both $0$-edges have indicator vectors with $z_1=1$), but $\rk_{B^L_{<3}(2)}(0)=\rk_{B^L(2)}(0)=2$.
This shows that an analogue of Theorem \ref{thm:logbest} cannot be proved for the truncated binarization. Moreover, since $B$ in the example is an affine binarization, also Lemma \ref{le:onlyaffine} cannot be extended in this context: however, we now show a slightly weaker version of it.
A binarization $B\subseteq\R\times[0,1]^d$ is \emph{linear} if $x=f(y)$ for all $(x,y)\in B$, where $f: \R^d\rightarrow \R$ is a linear function.

\begin{lemma}\label{lem:onlylinear}
Up to permuting variables, the logarithmic truncated binarization is the only linear binarization whose projection on the $y$-space is a truncated hypercube.
\end{lemma}

\begin{proof}
Let $2^{d-1}< v< 2^d$ for some $d$ and let $B$ be a linear binarization as in the statement. 
We would like to use the same argument as in the proof of Lemma \ref{le:onlyaffine}. However, since that proof exploits the symmetry of the hypercube with respect to variable permutation, here we need an ad-hoc argument for variable $x_d$.

Since $B$ is linear, there exist $a\in\R^d$ such that $x=ay$ for every $(x,y)\in B$.
We have that $\pi_x(V(B))=\{0,\dots,v-1\}$. Moreover, vertices $(a_i,\mathbf{e}_i)$ show that $a_1,\dots,a_d\ge1$. For $i\in \{0,\dots,v-1\}$, let $y^i$ be such that $(i,y^i)\in B$, and $X_i=\{j\in [d]: y^i_j=1\}$, so that $i=ay^i=a(X_i)$, where we write $a(X)=\sum_{j\in X} a_j$ for $X\subseteq [d]$. We have $a(X_i)\neq a(X_{i'})$ for any $i\neq i'$ in  $\{0,\dots,v-1\}$. 

We first show that $a_d=2^{d-1}$. Notice that there are $v-a_d$ numbers in $\{a_d,\dots, v-1\}$ and $v-2^{d-1}$ vertices $(x,y)\in V(B)$ with $y_d=1$, for which $x\geq a_d$: hence, $a_d\leq 2^{d-1}$. Moreover, if the inequality is strict, by the same counting argument we have that there is $i\in\{a_d,\dots,v-1\}$ with $y^i_{d}=y^{i-a_d}_d$, i.e., $d\not\in X_i\cup X_{i-a_d}$.
However, we have \[a(X_{i})-a(X_{i-a_d})=a_d \implies a(X_{i}\setminus X_{i-a_d})=a(X_{i-a_d}\setminus X_i)+a_d.\]
But now, since $d\not\in X_{i-a_d}$ we have $X_i\setminus X_{i-a_d}=(X_{i-a_d}\setminus X_i)\cup \{d\}$, hence $d\in X_{i}$, a contradiction.

Now one can show that $\{a_1,\dots,a_{d-1}\}=\{2^0,\dots,2^{d-2}\}$ exactly as in the proof of Lemma \ref{le:onlyaffine}, concluding the proof.
\end{proof}

In \cite{dash2018binary} it is shown that any affine binarization of a variable $x$ can be turned into a linear binarization of $x$ (with the same number of $y$-variables) through a unimodular transformation. This implies that affine and linear binarizations are equivalent in terms of split closure (see \cite{dash2018binary} for details). The example above, together with Lemma \ref{lem:onlylinear}, shows that unimodular transformations do not preserve the rank, hence using specific non-linear or non-affine binarizations might be beneficial in terms of efficiency of the sequential convexification procedure.

\section{Conclusion}\label{sec:conclusion}
Binary extended formulations featuring natural binarizations form a broad class that in our opinion deserves more attention. First, the vertices of such formulations admit a simple characterization (Theorem \ref{thm:verticesQ}), that allows a deeper understanding of the performance of sequential convexification on such formulations; second, natural binarizations lend themselves to a notion of rank (Definition \ref{def:rankB}) that allows to compare them, as we have done for the classical binarizations from the literature.

Given the variety of (natural) binarizations one could come up with, it is natural to ask which binarizations are ``better'' than others with respect to a given criterion. This question was investigated in \cite{dash2018binary}, which showed that certain binarizations with a large number of variables are optimal in terms of split closure. Our rank, stemming from a sequential convexification perspective, can be a valuable tool in investigating this question for binarizations with any number of variables. 

In light of Lemmas \ref{le:rank-unary}, \ref{le:rank-full}, one concludes that the unary binarization $B^U(d)$ is better than the full binarization $B^F(d)$ in terms of rank. We do not know whether $B^U(d)$ is optimal with respect to all binarizations that are isomorphic to a $d$-simplex. We leave this as our first open question.

Since the number of variables is an important factor for the efficiency of mixed-integer programming, it is interesting to focus on binarizations with the minimum number of variables, which is logarithmic in the length of the range of $x$.
Hypercube binarizations are a prominent example of this, and the logarithmic binarization $B^L$ is optimal among them (Theorem \ref{thm:logbest}). The rank function of $B^L$ is not comparable with the rank function of $B^U$, but it is definitely larger than the latter when one restricts to a small number of values: $\rk_{B^U}(\alpha)=1$ for any $\alpha$, while $\rk_{B^L}(\alpha)$ can be as large as $d$ depending on $\alpha$. This seems to suggest that there is a trade-off between the number of variables in a (natural) binarization and its rank. 

The situation is less clear for the truncated logarithmic binarizations $B^L_{<v}$. While the rank of $B^L_{<v}$ is at most the rank of $B^L$ (Corollary \ref{cor:trunc2}), it is not optimal: other isomorphic binarizations can have smaller rank (see the example at the end of Section \ref{sec:trunc}). It is not clear whether $B^L_{<v}$ (or another binarization with the same number of variables) could be preferable to $B^L$, which has a much simpler polyhedral structure. However it would be interesting to determine whether, for fixed $d$ and $v$, there is a binarization that is optimal with respect to the rank among those isomorphic to $B^L_{<v}(d)$. This is our second open question.

Finally, we mention another research direction that is, as far as we know, still unexplored. Given a binary extended formulation $Q$ of a polytope $P$, apply one round of the Sherali-Adams hierarchy \cite{sherali1990hierarchy} to the binary variables of $Q$: what can we say about the resulting extended formulation? We believe that the approach and the techniques developed in this paper will help answer this question and others in the same spirit.

\section*{Acknowledgments.}
Manuel Aprile and Marco Di Summa are supported by a grant SID 2019 of the University of Padova.

\printbibliography
\end{document}